\documentclass[12pt]{amsart}
\usepackage{amssymb,amsmath,amsthm}
\numberwithin{equation}{section}
\usepackage{enumerate}
\usepackage{hyperref}

\newtheorem{thm}{Theorem}[section]
\newtheorem{lemma}[thm]{Lemma}

\newtheorem{cor}[thm]{Corollary}
\newtheorem{prop}[thm]{Proposition}
\theoremstyle{definition}

\newtheorem{defn}[thm]{Definition}

\usepackage{color}
\begin{document}
\title[Smirnov classes]{The Smirnov classes for the Fock space and complete Pick spaces}

\author{Michael T. Jury}
\address{University of Florida}
\email{mjury@ufl.edu}

\author{Robert T.W. Martin}
\address{University of Cape Town}
\email{rtwmartin@gmail.com}

\thanks{Second author acknowledges support of NRF CPRR Grant 105837.}
\date{\today}

\begin{abstract}
For a Hilbert function space $\mathcal H$ the Smirnov class $\mathcal N^+(\mathcal H)$ is defined to be the set of functions expressible as a ratio of bounded multipliers of $\mathcal H$, whose denominator is cyclic for the action of $Mult(\mathcal H)$. It is known that for spaces $\mathcal H$ with complete Nevanlinna-Pick (CNP) kernel, the inclusion $\mathcal H\subset \mathcal N^+(\mathcal H)$ holds.  We give a new proof of this fact, which includes the new conclusion that every $h\in\mathcal H$ can be expressed as a ratio $b/a\in\mathcal N^+(\mathcal H)$ with $1/a$ already belonging to $\mathcal H$. 

The proof for CNP kernels is based on another Smirnov-type result of independent interest. We consider the Fock space $\mathfrak F^2_d$ of free (non-commutative) holomorphic functions and its algebra of bounded (left) multipliers $\mathfrak F^\infty_d$. We introduce the (left) {\em free Smirnov class} $\mathcal N^+_{left}$  and show that every $H \in \mathfrak F^2_d$ belongs to it. The proof of the Smirnov theorem for CNP kernels is then obtained by lifting holomorphic functions on the ball to free holomorphic functions, and applying the free Smirnov theorem. 

\end{abstract}

\maketitle
\section{Introduction}
\subsection{The Smirnov class}
The {\em Smirnov class} $N^+(\mathbb D)$ is the collection of all holomorphic functions $h$ in the open unit disk $\mathbb D\subset \mathbb C$ of the form
\begin{equation}
  h=\frac{b}{a}
\end{equation}
where $a,b$ are bounded analytic functions, and $a$ is an {\em outer} function; this means that the set of functions $\{af:f\in H^2\}$ is dense in the Hardy space $H^2(\mathbb D)$ (equivalently, the vector $a\in H^2$ is a cyclic vector for the unilateral shift operator $f(z)\to zf(z)$).  As a consequence of Riesz's theorem on inner-outer factorizations, all of the Hardy spaces $H^p(\mathbb D)$, $p>0$, are contained in $N^+$. (See e.g. \cite[Section 2.5]{duren-book}.) In the case $p=2$, the Hilbert space structure allows one to obtain Smirnov representations $h=\frac{b}{a}$ where the $a,b$ can be chosen to satisfy certian additional conditions. In particular, Sarason \cite[Proposition 3.1]{sarason-2008} observed that for $h\in H^2$, we can choose $a,b$ so that additionally
\begin{equation}\label{eqn:sarason-inner}
  |a(z)|^2 +|b(z)|^2=1 \quad \text{for a.e. } |z|=1.
\end{equation}
Indeed, since $\log(1+|h|^2)$ is integbable on the unit circle, from the theory of inner-outer factorizations there exists an $H^\infty$ outer function $a$ so that
\begin{equation}
  |a(z)|^2=\frac{1}{1+|h(z)|^2} \quad \text{for a.e. } |z|=1,
\end{equation}
taking this for $a$ and putting $b=ah$ gives the desired pair. We observe that in this construction, the function $1/a$ already belongs to $H^2(\mathbb D)$, and 
\begin{equation}
  \left\| \frac{1}{a}\right\|^2 = 1+\|h\|^2.
\end{equation}
We also remark that one way to interpret the condition (\ref{eqn:sarason-inner}) is that the map
\begin{equation}
  f\to \begin{pmatrix} af \\ bf\end{pmatrix}
\end{equation}
is an isometric multiplier from $H^2$ into $H^2\oplus H^2$.

In this paper our goal is to prove two analogs of this result of Sarason, one in the setting of so-called ``free holomorphic functions'', which will be discussed shortly, and one for reproducing kernel Hilbert spaces possessing a (normalized) {\em complete Nevanlinna-Pick} (CNP) kernel. 

Following \cite{AHMR-2018}, for any reproducing kernel Hilbert space $\mathcal H$ with multiplier algebra $Mult(\mathcal H)$, we define the {\em $\mathcal H$-Smirnov class} to be
\begin{equation}
  \mathcal N^+_{\mathcal H} =\left\{ \left. \frac{b}{a} \right| \ a,b\in Mult(\mathcal H), \text{ with } a  \text{ an }{\mathcal H} -\text{outer function } \right\}
\end{equation}
Here $Mult(\mathcal H)$ denotes the algebra of functions which multiply $\mathcal H$ into itself boundedly, and we say that $a$ is $\mathcal H$-outer if $a\mathcal H$ is dense in $\mathcal H$.   

Of particular importance to us will be the {\em Drury-Arveson space} $H^2_d$, which is the space of holomorphic fucntions in the unit ball of $\mathbb C^d$, with reproducing kernel $k(z,w)=(1-zw^*)^{-1}$. For this space we will prove the following theorem:

\begin{thm}\label{thm:commutative-main} Let $h\in H^2_d$. Then there exist $a,b\in Mult(H^2_d)$ such that
  \begin{itemize}
\item[i)] $a$ is $H^2_d$-outer, 
\item[ii)]
 \begin{equation}
    h(z)=\frac{b(z)}{a(z)},
  \end{equation}
\item[iii)] the column $\begin{pmatrix} a \\ b\end{pmatrix}$  is a contractive multiplier, and
\item[iv)] $1/a\in H^2_d$,  with $\|1/a\|^2_{H^2_d}\leq 1+\|h\|^2_{H^2_d}$. 
\end{itemize}
\end{thm}

It is already known \cite[Theorem 10.3]{alpay-kap-bol} that every $H^2_d$ function is $H^2_d$-Smirnov (so that (i) and (ii) hold; in which case (iii) can be obtained by scaling). The new part of Theorem~\ref{thm:commutative-main} is (iv). (It should be noted that the construction of Smirnov representations of $H^2_d$ functions we carry out here is essentially different from those in \cite{{alpay-kap-bol}} or \cite{AHMR-2017-2}; generically the conclusion (iv) does not hold in those constructions.)

We will also extend Theorem~\ref{thm:commutative-main} to a more general class of spaces. A (normalized) {\em complete Nevanlinna-Pick (CNP)} kernel on a set $\Omega$ is a function $k:\Omega\times\Omega\to \mathbb C$ of the form
\begin{equation}
  k(x,y)=\frac{1}{1-u(x)u(y)^*}
\end{equation}
where $u$ is any function from $\Omega$ into the open unit ball $\mathbb B^d$ (here $d=\infty$ is allowed; in this case $u$ would be a map from $\Omega$ into $\ell^2(\mathbb N)$ satisfying $\|u(x)\|^2=\sum_{n=1}^\infty |u_n(x)|^2<1$ for all $x\in\Omega$). From a basic result of Agler and McCarthy \cite{agler-mccarthy-2000}, if we set $E=u(\Omega)\subset \mathbb B^d$ and define
\begin{equation}
  \mathcal H_E :=cl(span\{(1-zw^*)^{-1}:w\in E\})\subset H^2_d, \label{HE}
\end{equation}
then $\mathcal H(k)$ is isometrically isomorphic to $\mathcal H_E$ via the map $k(\cdot, y)\to (1-zu(y)^*)^{-1}$. Moreover the restriction map $f\to f|_E$ is a co-isometry of $H^2_d$ onto $\mathcal H_E$, and the restriction $b\to b|_E$ is a (complete) contraction of $Mult(\mathcal H^2_d)$ onto $Mult(\mathcal H_E)$.  Finally, functions in $\mathcal H_E, Mult(\mathcal H_E)$ respectively have norm-preserving extensions to functions in $H^2_d, Mult(H^2_d)$ respectively.  An analog of (\ref{thm:commutative-main}) for spaces with a CNP kernel will follow readily from these observations.  As for the special case $H^2_d$, it is already known that $\mathcal H\subset \mathcal N^+(\mathcal H)$ holds for all normalized CNP spaces \cite{AHMR-2017}; the new conclusion here will be that $1/a\in\mathcal H$. 

\subsection{Free holomorphic functions} 
 Theorem~\ref{thm:commutative-main} will be proved, via lifting arguments, from a corresponding ``free'' version. To state it, we first quickly recall some basic definitions and results about noncommutative power series.

By the {\em free unit ball }$\mathcal B_d$ we mean the set of all $d$-tuples $X=(X_1, \dots X_d)$ of $n\times n$ matrices, over all sizes $n$, which obey the norm estimate
\begin{equation}
  \left\| \sum_{j=1}^d X_jX_j^* \right\|<1.
\end{equation}
(Throughout, we will write such expressions as though $d$ is finite, but we will allow $d=\infty$, that is, infinite sequences, as this will be important for our applications). 

We write $\mathbb F_d^+$ for the free semigroup on $d$ letters $\{1, 2, \dots \}$ (again, $d$ countably infinite is allowed); that is, the set of all words $\alpha= i_1i_2\cdots i_k$, over all (finite) lengths $k$, where each $i_j\in\{1, 2, \dots\}$. We write $|\alpha|=k$ for the length of $\alpha$. We also include the {\em empty word} in $\mathbb F_d^+$, and denote it $\varnothing$, and put $|\varnothing|=0$.  There is a {\em transpose map} which reverses the letters in a word: if $\alpha=i_1i_2\cdots i_k$ then we write
\begin{equation}
  \alpha^\dag = i_ki_{k-1}\cdots i_2i_1.
\end{equation}

For a $d$-tuple of $n\times n$ matrices $X=(X_1, X_2, \dots X_d)$ we write, for each word $\alpha = i_1i_2\cdots i_k$ in $\mathbb F_d^+$, 
\begin{equation}
  X^\alpha := X_{i_1}X_{i_2}\cdots X_{i_k},
\end{equation}
and declare $X^\varnothing =I_n$. By a {\em free power series} we mean an expression of the form
\begin{equation}\label{free-series-def}
  F(X)\sim \sum_{\alpha\in \mathbb F_d^+} c_\alpha X^\alpha
\end{equation}
By a {\em free holomorphic function} in the free unit ball $\mathcal B_d$ we mean a free power series with the property that the sum
\begin{equation}
  \sum_{k=0}^\infty \left\| \sum_{|\alpha|=k} c_\alpha X^\alpha\right\|
\end{equation}
converges for all $X\in\mathcal B_d$. In this case, for each $X\in \mathcal B_d$, of size $n\times n$, the expression (\ref{free-series-def}) converges in the norm of $M_n$ to some matrix, which we denote $F(X)$. (One thinks of $F$ as a ``graded'' function: For each fixed input $d-$tuple, $X = (X_1, ... , X_d)$, of $n\times n$ matrices, $F$ outputs the  $n\times n$ matrix $F(X)$ \cite{KVV,BMV}. For example, if $X=(0_n,0_n,\dots 0_n)$, then $F(X)= c_\varnothing I_n$.) 

The {\em Fock space} $\mathfrak F^2_d$ consists of all free power series whose coefficients are square-summable: $\sum_{\alpha\in\mathbb F_d^+} |c_\alpha|^2$. The Fock space is a Hilbert space with the $\ell^2$ inner product. Every such series represents a free holomorphic function in $\mathcal B_d$. We let $\mathcal F_d^\infty$ denote the set of all bounded free holomorphic functions in $\mathcal B_d$, that is, 
\begin{equation}
  \mathfrak F_d^\infty =\{F: \sup_{X\in\mathcal B_d} \|F(X)\|<\infty\}
\end{equation}
 $\mathfrak F_d^\infty$ is a Banach algebra when equipped with the supremum norm.  It is a fact that if $F\in \mathfrak F_d^\infty$ and $G\in \mathfrak F^2_d$, then the product $F(X)G(X)$ (defined by the evident Cauchy product of free series) also belongs to $\mathfrak F^2_d$, and the operator sending $G$ to $FG$ is bounded, with norm equal to $\|F\|_\infty$. We call such functions $F$ {\em left multipliers} of $\mathfrak F^2_d$.  Similarly, one can consider free functions which act by bounded right multiplication, called {\em right multipliers} and it is a fact that a series
\begin{equation}
  F(X)=\sum_{\alpha\in \mathbb F^+_d} c_\alpha X^\alpha
\end{equation}
 is a bounded left multiplier if and only if 
\begin{equation}
  F^\dag(X)=\sum_{\alpha\in \mathbb F^+_d} c_{\alpha^\dag} X^\alpha
\end{equation}
is a bounded right multiplier.

\begin{defn}\label{def:free-smirnov}
The {\em left Smirnov class} $\mathcal N^+_{left}(\mathcal B_d)$ is defined to be the set of all free
functions $H(X)$ defined in the free ball, which can be written as
\begin{equation}
  H(X) =B(X)A(X)^{-1}  
\end{equation}
where $A,B$ are bounded left multipliers and $A$ is a left outer
function. (This means that left multiplication by $F\to AF$ has dense range
in $\mathfrak F^2_d$.)

\end{defn}

\begin{defn}
  A pair of left multipliers $A,B$ is called a (left) {\em inner-outer pair} if the $A$ is left outer, and the column
  \begin{equation}
    \begin{pmatrix} A \\B\end{pmatrix}
  \end{equation}
is an isometric (left) multiplier of the Fock space $\mathfrak F^2_d$ into $\mathfrak F^2_d \oplus \mathfrak F^2_d$. 
\end{defn}

We then have the following ``free'' Smirnov theorem: 

\begin{thm}\label{thm:free-main} 
  Let $H(X)$ be a free holomorphic function in the free unit ball. Then $H$ belongs to the left Smirnov class $\mathcal N^+_{left}(\mathcal B_d)$ if and only if left  multiplication by $H$ is a closed, densely defined operator in $\mathfrak F^2_d$.  In this case, there exist left multipliers $A, B\in\mathfrak F^\infty_d$ such that:
  \begin{itemize}
  \item[i)] $A$ is a left outer function, 
\item[ii)] For all $X$ in the free unit ball,
\begin{equation}\label{eqn:factor-main}
    H(X) = B(X)A(X)^{-1},
  \end{equation}
\item[iii)] The column
  \begin{equation}
    \begin{pmatrix} A\\ B\end{pmatrix} \quad  \mbox{is an isometric (left) multiplier.}
  \end{equation}
  \end{itemize}
Moreover, the pair $A,B$ satisfying (i)-(iii) is unique up to multiplication by a unimodular scalar.

In addtion, every $H\in \mathfrak F^2_d$ belongs to the left Smirnov classes, and in this case the factorization (\ref{eqn:factor-main}) has the property that the function $A^{-1}$ belongs to $\mathfrak F^2_d$ and $\|A^{-1}\|^2 = 1+\|H\|^2$. 
\end{thm}

The first part of this theorem is seen to be a free analogue of another result of Sarason \cite[Lemma 5.2]{sarason-2008}, which says that a function $h$ belongs to the Smirnov class $N^+(\mathbb D)$ if and only if mulitplication by $h$ has dense domain in $H^2(\mathbb D)$, in which case $h$ can be expressed (essentially uniquely) as a quotient $b/a$ with $a$ outer and $|a|^2+|b|^2=1$ a.e. on the circle. We will call the pair $A,B$ in the conclusion of the theorem the {\em canonical (left) representation} of the (left) Smirnov function $H$.

\section{Free Smirnov functions}\label{sec:free-smirnov}

We begin this section by carefully enumerating the results about free holomorphic functions that we require. Most of the basic results we will use may be found in Popescu \cite{popescu-2006}. While the results are stated there for a finite number of free variables $X_1, \dots X_d$, it is a straighforward matter to check that the proofs hold also for countably many variables. 

The {\em Fock space} $\mathcal F^2_d$ is a Hilbert space with orthonormal basis $\{\xi_\alpha\}_{\alpha\in\mathbb F^+_d}$ labeled by the free semigroup on $d$ letters. For each letter $i$ there is an isometric operator $L_i$ acting in $\mathcal F^2_d$ defined on the orthonormal basis $\{\xi_\alpha\}$ by
  \begin{equation}
    L_i\xi_\alpha = \xi_{i\alpha}
  \end{equation}
(called the {\em left creation operators}), here $i\alpha$ just means the word obtained from $\alpha$ by appending the letter $i$ on the right. One analogously defines the right creation operators $R_i$. The operators $L_i$ (and the $R_i$) have orthogonal ranges, and hence obey the identity
  \begin{equation}
    L_i^*L_j=\delta_{ij}I.
  \end{equation}
The {\em free semigroup algebra} $\mathcal L_d$ is the (unital) WOT-closed algebra generated by the $L_i$. For a word $\alpha =i_1i_2\dots i_k$ we write $L_\alpha=L_{i_1}L_{i_2}\cdots L_{i_k}$. One can show that the corresponding algebra generated by the right creation operators, $\mathcal R _d$, and $\mathcal L _d$ are each other's commutants (and in fact that $\mathcal R _d$ is simply the image of $\mathcal L _d$ under conjugation by the `tranpose unitary': precisely, if $W:\mathfrak F^2_d\to \mathfrak F^2_d$ is the unitary map which acts on basis vectors as $W\xi_\alpha =\xi_{\alpha^\dag}$, then we have $WL_\alpha W^* = R_{\alpha^\dag}$).  Each element $F$ of the free semigroup algebra admits a Fourier-like expansion
\begin{equation}
  F\sim \sum_{\alpha\in\mathbb F^+_d} c_{\alpha} L_\alpha,
\end{equation}
and the Ces\`aro means of this series converge in the strong operator topology (SOT) to $F$ \cite{DPac}. 
These objects are related to the free holomorphic functions described in the introduction, as follows:
\begin{itemize}
\item[1)] \cite[Theorem 1.1]{popescu-2006} If $\sum_{\alpha\in\mathbb F^+_d}c_\alpha\xi_\alpha\in\mathfrak F^2_d$, then for each $X=(X_1, \dots ,X_d)$ in the free unit ball $\mathcal B_d$, the series
  \begin{equation}
     \sum_{k=0}^\infty \left\| \sum_{|\alpha|=k} c_\alpha X^\alpha\right\|
  \end{equation}
converges. Hence $H(X)=\sum_{\alpha\in\mathbb F^+_d} c_\alpha X^\alpha$ defines a free holomorphic function in $\mathcal B_d$, and we may identify the Fock space with the space of free holomorphic functions with $\ell^2$ coefficients.  
\item[2)] \cite[Theorem 3.1]{popescu-2006} A free power series $F(X) =\sum_{\alpha\in\mathbb F^+_d} c_\alpha X^\alpha$ satisfies $\sup_X \|F(X)\|<\infty$ if and only if the sum $\sum_{\alpha\in\mathbb F^+_d} c_{\alpha} L_\alpha$ is the Fourier expansion of an element $F\in\mathcal L_d$, and moreover $\|F\|_{\mathcal L_d} =\sup_{X\in\mathcal B_d}\|F(X)\|$.  Thus we may identify $\mathcal L_d$ with $\mathfrak F^\infty_d$. (In fact this identification is completely isometric). 
\item[3)] \cite[Theorem 1.3]{popescu-2006} For any $F$ in $\mathfrak F^2_d$ or $\mathfrak F^\infty_d$, the map $F\to F(X)$ is continuous into $M_n$, for each $X\in \mathcal B_d$. 
\item[4)] \cite[Theorem 1.4]{popescu-2006} If $F$ and $G$ are free holomorphic functions, expressed as power series
  \begin{equation}
    F(X) = \sum_{\alpha\in\mathbb F^+_d} f_\alpha X^\alpha, \quad H(X) = \sum_{\alpha\in\mathbb F^+_d} g_\alpha X^\alpha,
  \end{equation}
then the product $H(X)=F(X)G(X)$ is expressible as a free power series, with coefficients
\begin{equation}
  h_\alpha = \sum_{\beta\gamma =\alpha} f_\beta g_\gamma.
\end{equation}
In particular, the action of an element $F$ of the free semigroup algebra on the Fock space may be interpreted as left mulitplication by the free holomorphic function $F(X)$. 
\end{itemize}

Recall now the definitions of the left Smirnov class given in the introduction \ref{def:free-smirnov}, and the definition of free outer functions. 

It is evident that if $H$ is a free holomorphic function belonging to the left Smirnov
class, then left multiplication by $H$ gives a densely defined
operator in $\mathfrak F^2_d$.  Indeed, if $H(X)=B(X)A(X)^{-1}$ is
left Smirnov, then the set of free functions $A(X)F(X)$, where
$F\in\mathfrak F^2_d$, is dense in $\mathfrak F^2_d$ (since $A$ is
assumed free outer), and hence for all $F\in\mathfrak F^2_d$ we have
\begin{equation}
  HAF= BA^{-1}AF= BF\in\mathfrak F^2_d.   
\end{equation}
Moreover, it is straightforward to verifty that if $H$ is left
Smirnov, and we define
\begin{equation}
    \mathcal D_H = \{F\in \mathfrak F^2_d : HF\in \mathfrak F^2_d\}
\end{equation}
then the densely defined operator $T_H:\mathcal D\to \mathfrak F^2_d$ given by
$T_HF=HF$ is {\em closed}. (See Lemma~\ref{lem:free-left-multipliers} below.) It is also evident
that $\mathcal D_H$ is invariant for the algebra of (bounded) right mulitplication
operators $\mathcal R_d$. The main goal of this section is to prove a converse to
this set of statements. Namely, we will show that if $H$ is a free holomorphic function
and left multiplication by $H$ gives a closed, densely defined
operator in $\mathfrak F^2_d$, then $H$ is a left Smirnov
function.
%


\begin{lemma}\label{lem:outer-invertible}
  If the free holomorphic function $A\in\mathfrak F^\infty_d$ is left outer, then for each $X$ in the free unit ball, $A(X)$ is invertible. 
\end{lemma}

\begin{proof}
  Suppose $A(X)$ is not invertible for some $X$. Then there exists a nonzero vector $v\in \mathbb C^g$ such that $v^*A(X)=0$.  Write $M_A$ for the operator of left multiplication by $A$. For any $F\in \mathfrak F_d^2$, we then have
  \begin{equation}
    v^*(AF)(X)= v^*A(X)F(X)=0.
  \end{equation}
which shows that every function $G$ in $\mathrm{Ran} (M_A)$ also has $v^*G(X)=0$. But then the map 
\begin{equation}
e_{X,v}:F\to v^*F(X)  
\end{equation}
is a continuous, nonzero linear map from $\mathfrak F^2_d$ into $\mathbb C^n$, which annihilates the range of $M_A$, whence this range is not dense and $A$ is not left outer. 
\end{proof}

We now consider certain unbounded operators in $\mathfrak F^2_d$ induced by left mulitplication by free holomorphic functions. 

\begin{defn}
  A closed, densely-defined linear operator $T:Dom(T)\to \mathfrak F_d^2$ {\em commutes with the right free shift $R=(R_1, ..., R_d)$} if for each $j=1, \dots d$
  \begin{itemize}
  \item[i)] $R_j Dom(T)\subset Dom(T)$, and 
  \item[ii)] $R_jT=TR_j$ on $Dom(T)$.  
  \end{itemize}
\end{defn}

\begin{defn}
  Let $T$ be a closed, densely defined operator in $\mathfrak F_d^2$.  Say $T$ is {\em local} if $Dom(T)$ has the following property: whenever $F\in Dom(T)$ satisfies $\widehat{F}(0):=\langle F, \xi_\varnothing\rangle =0$, then for each $j$ the backward shifts $R_j^*F$ belong to $Dom(T)$.
\end{defn}

\begin{lemma}\label{lem:free-left-multipliers}
  Let $H$ be a free holomorphic function in the ball, with free power series expansion
  \begin{equation}
    H(X) = \sum_{\alpha} h_\alpha X^\alpha.
  \end{equation}
 Let $\mathcal D_H$  be the vector space of all Fock space functions $F\sim \sum_{\alpha\in\mathbb F^+_d} f_\alpha \xi_\alpha$ with the property that the series
 \begin{equation}
   HF:= \sum_{\alpha}\left\{\sum_{\beta\gamma=\alpha} h_\beta f_\gamma\right\}\xi_\alpha
 \end{equation}
defines an element of $\mathfrak F^2_d$. Assume that $\mathcal D_H$ is dense in $\mathfrak F_d^2$. Then the (densely defined) operator $T_H:\mathcal D_H\to \mathfrak F_d^2$ given by
\begin{equation}
  T_H F :=HF
\end{equation}
is closed, local, and commutes with the right free shift. 
\end{lemma}

\begin{proof} First, as is immediately seem by inspecting the free power series expansions, if $F\in\mathfrak F_d^2$ and $HF\in \mathfrak F_d^2$, then also $(R_jF)H=R_j(HF)\in\mathfrak F_d^2$, which shows that $T_H$ commutes with $R$. To see that $T_H$ is local, again directly from the power series one sees that if $F\in\mathcal D_H$  has no constant term (i.e. $f_\varnothing=0$), so that $F = \sum_j R_jR_j^*F$, then we have for each $j$
  \begin{equation}
   H (R_j^*F) = R_j^*(HF).
  \end{equation}
Since the right-hand side belongs to $\mathfrak F_d^2$, so does the left, hence $R_j^*F\in\mathcal D_H$ for each $j$. 

Finally, let us show that $T_H$  is closed. Suppose $(F_n)$ is a sequence from $\mathcal D_H$ with $F_n\to F$ in $\mathfrak F^2_d$ and $T_HF_n =HF_n \to G$ for some $G\in \mathfrak F_d^2$. Then, in particular, these series converge pointwise as free holomorphic functions, i.e. for all $X$ in the free unit ball we have $F_n(X)\to F(X)$ and $H(X)F_n(X)\to G(X)$ in matrix norm.  But then it follows that $H(X)F_n(X)\to H(X)F(X)$ for all $X$; therefore $H(X)F(X)=G(X)$ for all $X$, which means that $HF=G\in\mathfrak F_d^2$. This proves simultaneously that $F\in \mathcal D_H$ and $T_HF =G$, so that $T_H$ is closed. 
\end{proof}

We refer to $H$ as a {\em (closed) densely defined left multiplier} of $\mathfrak F_d^2$, unless otherwise noted we will always assume $H$, viewed as a left multiplier, has domain $\mathcal D_H$ as described in the lemma.

For any operator $T$ defined in $\mathfrak F_d^2$, as usual the graph of $T$ is the subspace
\begin{equation}
  G(T) = \left\{ \begin{pmatrix} F \\TF \end{pmatrix} : F\in Dom(T)\right\} \subset \mathfrak F_d^2 \oplus \mathfrak F_d^2.
\end{equation}
Recall that by definition the operator $T$ is closed if and only if $G(T)$ is a closed subspace.


We can now prove Theorem~\ref{thm:free-main}.

\begin{proof}[Proof of Theorem~\ref{thm:free-main}]
  Let $H$ be a free holomorphic function and suppose the operator $T:F\to HF$ is densely defined. As above we take $Dom(T)=\{F\in\mathfrak F^2_d:HF\in\mathfrak F^2_d\}$. By Lemma~\ref{lem:free-left-multipliers}, $T$ is closed and commutes with the right free shift, which means that its graph
  \begin{equation}
        G(T) = \left\{ \begin{pmatrix} F \\ TF\end{pmatrix} : F\in Dom(T) \right\} \subset \mathfrak F_d^2 \oplus \mathfrak F_d^2
  \end{equation}
is a closed, invariant subspace for the operators $R_j\oplus R_j$ in $\mathfrak F_d^2 \oplus \mathfrak F_d^2$, for each $j$. 
We first examine the {\em wandering
subspace} for the restriction of $R\oplus R$ to $G(T)$, 
\begin{equation}
  \mathcal W = G(T)\ominus (R\oplus R)G(T).
\end{equation}
(Here $ (R\oplus R)G(T)$ is a shorthand for the (closed) span of the ranges of the operators $R_j\oplus R_j$ restricted to $G(T)$. )

We claim this space $\mathcal W$ is one-dimensional.  To see this, consider the map $Q:G(T)\to \mathfrak F^2_d$ defined by
  \begin{equation*}
    Q(F\oplus TF) = F.
  \end{equation*}
Evidently $Q$ is contractive, and $Q$ has dense range, since its range is precisely $Dom(T)$. Hence $Q^*:\mathfrak F^2_d\to G(T)$ is injective. 
In addition, $Q$ intertwines the action of $R$ in $\mathfrak F_d^2$ and the restriction
of $R\oplus R$ to $G(T)$.  We will show that $\mathcal W$ is spanned by the vector $Q^*\xi_\varnothing$.

First we claim that $Q^*\xi_\varnothing\in \mathcal W$. Observe that $Q^*\xi_\varnothing \neq 0$ since $Q^*$ is injective. Next, for any
$F\oplus TF\in G(T)$ and any nonempty word $\alpha$, 
\begin{align*}
  \langle Q^*\xi_\varnothing , R_\alpha F\oplus
  R_\alpha TF\rangle &= \langle \xi_\varnothing , Q(R_\alpha F\oplus TR_\alpha F)\rangle \\
&= \langle\xi_\varnothing, R_\alpha F\rangle \\
&=0.
\end{align*}
Now suppose $F\oplus TF$ is a vector in $\mathcal W$ which is orthogonal to
$Q^*\xi_\varnothing$, we show $F\oplus TF$ is $0$. By assumption
\begin{align*}
  0 &=\langle F\oplus TF, Q^*\xi_\varnothing\rangle \\
&= \langle F, \xi_\varnothing\rangle.
\end{align*}
So $F\in Dom(T)$ and $F(0)=0$, so by the assumption that $T$ is local,  we have $R_j^*F\in
Dom(T)$ for each $j$. We can then write $F=\sum R_jR_j^*F$, and we have
\begin{align*}
 F\oplus TF &=  \left(\sum_j R_jR_j^*F\right)\oplus T\left(\sum R_jR_j^*F \right)\\
 &= \sum (R_j\oplus R_j)(R_j^*F \oplus TR_j^*F)\\
\end{align*}
But this says $F\oplus TF\in \text{span}_j(R_j\oplus R_j)G(T)=\mathcal W^\bot$.  As we
assumed $F\oplus TF\in \mathcal W$, we conclude $F=0$.

Now, by the Davidson--Pitts version of the Beurling theorem for $\mathcal R_d$ (stated in \cite[Theorem 2.1]{davidson-pitts-1999} for $\mathcal L_d$, but it is evident that the analogous statements hold for $\mathcal R_d$), there is an isometric left multiplier
  $M_\Phi= \begin{pmatrix} M_A \\ M_B\end{pmatrix}$ from $\mathfrak F_d^2$ onto $G(T)$, intertwining the action of $R$ and $R\oplus R$, and taking the wandering vector $\xi_\varnothing$ for $R$ in $\mathfrak F_d^2$ onto the (unit) wandering vector $\frac{Q^*\xi_\varnothing}{\|Q^*\xi_\varnothing\|}$ in $\mathcal W\subset G(T)$, so that 
  \begin{equation}
    G(T) = \mathrm{Ran} (M_\Phi ) = \{ \begin{pmatrix} AF \\ BF\end{pmatrix}
    :F\in\mathcal F^2\}.
  \end{equation}
This pair is the desired $A,B$: we have just seen that the column $\begin{pmatrix} A \\ B\end{pmatrix}$ is isometric. Further,
$A$ is left outer, since this expression just given for $G(T)$ shows that $Dom(T)=\{ AF:F\in\mathfrak F_d^2\}$ so that the latter space is dense. Finally, since the column $\begin{pmatrix} A \\ B\end{pmatrix}$ itself belongs to the graph of $T$, we must have $B=HA$. By Lemma~\ref{lem:outer-invertible}, $A(X)$ is invertible for each $X\in\mathcal B_d$, and hence $H(X)=B(X)A(X)^{-1}$. 

Now let us prove uniqueness. Besides the pair $A,B$ already constructed, we suppose $A_1, B_1$ is another pair satisfying (i)-(iii). Then the column $\Phi_1=\begin{pmatrix} A_1 \\ B_1\end{pmatrix}$ is an isometric left multiplier, whose range is a (necessarily closed) $R\oplus R$-invariant subspace $\mathcal M\subset G(T)$. Yet another application of the Beurling theorem for $R$ shows that there is an isometric left multiplier $\Theta\in\mathcal L_d$ such that $A_1=\Theta A, B_1=\Theta B$. Since $A_1$ has dense range by hypothesis, we conclude that $\text{ran}\Theta$ must be dense in $\mathfrak F^2_d$, but since $\Theta$ is already an isometry, $\Theta$ must in fact be unitary. But the only unitaries which commute with $\mathcal R_d$ are unimodular scalars \cite[Corollary 1.5]{davidson-pitts-1999}, and thus $A,B$ are unique up to a unimodular constant. 

To finish the proof, let us now suppose $H\in \mathfrak F^2_d$. It follows that left multiplier $F\to HF$ is densely defined, since for example $HF$  belongs to $\mathfrak F^2_d$ for every free polynomial $F$. Thus by the first part of the theorem, $H$ is left Smirnov. Moreover, for the pair $A,B$ constructed above, we know that the range of the left free inner function
  \begin{equation}\label{rightAB}
    F\to \begin{pmatrix} AF \\ BF\end{pmatrix}
  \end{equation}
is precisely the graph of the operator of left multiplication by $H$. Since $H\in \mathfrak F^2_d$, the vector $\xi_\varnothing$ (corresponding to the constant free function $I$) belongs to $\mathcal D_H$, and hence
\begin{equation}
  \begin{pmatrix} I \\ H\end{pmatrix}
\end{equation}
belongs to the graph of $T_H$. Thus there is an $F\in \mathfrak F^2_d$ with $AF=I$, and hence $F=A^{-1}$ belongs to $\mathfrak F^2_d$. Finally since the square norm of the column $  \begin{pmatrix} I \\ H\end{pmatrix}$ is $1+\|H\|^2$, and the map (\ref{rightAB}) is isometric, it follows that in fact $\|A^{-1}\|^2 =1+\|H\|^2$.

\end{proof}

\section{$H^2_d$-Smirnov functions}\label{sec:h2d-smirnov}

We recall the connection between the free function spaces $\mathfrak F^2_d$, $\mathfrak F^\infty_d$, and the Drury-Arveson space $H^2_d$ and its multiplier algebra. 

First, given a free holomorphic function $H\in\mathfrak F^2_d$, we may of course evaluate it on a tuple of $1\times 1$ matrices $z=(z_1, z_2, \dots)$ satisfying $\sum_j |z_j|^2<1$, (i.e. a point of the open unit ball $\mathbb B^d$). The resulting holomorphic function $h(z)=H(z)$ belongs to the Drury-Arveson space on $\mathbb B^d$, and in fact this map is a co-isometry. In particular every $h\in H^2_d$ has a {\em free lift} to a free function $H\in \mathfrak F^2_d$, and there is a unique such lift preserving the norm: $\|H\|=\|h\|$. Namely, if $\mathbb{N} ^d$ denotes the additive monoid of $d-$tuples of non-negative integers, and $\mathbf{n} := (n_1 , ... , n_d ) \in \mathbb{N} ^+ _d$, set $z^\mathbf{n} := z_1 ^{n_1} \cdots z_d ^{n_d}$. Since the free semigroup $\mathbb{F} _d ^+$ is the universal monoid on $d$ generators, there is a unital semi-group epimorphism, $\lambda : \mathbb{F} _d ^+ \rightarrow \mathbb{N} _d ^+$, the \emph{letter counting map}, defined by $\lambda (\alpha) = (n _1 , ..., n_d )$, where $n_k$ is the number of times the letter $k$ appears in the word $\alpha$. Every $h \in H^2 _d$ has a Taylor series expansion (about $0$) indexed by $\mathbb{N} ^+ _d$, and the map:
$$ h(z) = \sum _{\mathbf{n} \in \mathbb{N} ^+ _d  } h_{\mathbf{n}} z^\mathbf{n} \  \mapsto \  H(Z) := \sum h_{\mathbf{n}} Z ^{\mathbf{n}}; \quad Z \in \mathcal{B} _d $$ 
defines an isometric embedding of $H^2 _d$ into $\mathfrak{F} ^2 _d$ (\emph{i.e.} $H^2 _d$ is identified with symmetric Fock space), where 
$$ Z ^{\mathbf{n}} := \sum _{\alpha | \ \lambda (\alpha ) = \mathbf{n}} Z^\alpha, $$ see \cite[Section 4]{Sha2013}. 
Likewise, the map $F\to F(z)$ is a completely contractive homomorphism from $\mathfrak F^\infty_d$ onto the multiplier algebra $Mult(H^2_d)$, (see \cite[Theorem 4.4.1, Subsection 4.9]{Sha2013} or \cite[Section 2]{DP-NP}) and again (by commutant lifting) every $f\in Mult H^2_d$ has a norm-preserving free lift to some $F\in \mathfrak F^\infty_d$ \cite{Ball2001-lift,DL2010commutant}. (Unlike the Hilbert space case, however, norm-preserving free lifts from $Mult(H^2_d)$ to $\mathfrak F^\infty_d$ may not be unique \cite[Corollary 7.1]{JMfree}.)

We need one more lemma:

\begin{lemma}\label{lem:free-outer-implies-outer}
  If the free function $A\in\mathfrak F^\infty_d$ is left outer, then the multiplier $a(z)=A(z)$ is outer for $H^2_d$. 
\end{lemma}

\begin{proof}
  If $a(z)$ is not outer for $H^2_d$, then there exists a nonzero continuous linear functional $\tau$ on $H^2_d$ which annihilates the range $aH^2_d$.  Since the symmetrization map $\sigma:\mathfrak F_d^2 \to H^2_d$ carries $A\mathfrak F_d^2$ onto $aH^2_d$, it follows that $\tau\circ \sigma$ is a nonzero, continuous linear functional on $\mathfrak F_d^2$ which annihilates $A\mathfrak F_d^2$, so that this space is not dense in $\mathfrak F_d^2$, and $A$ is not outer. 
\end{proof}
(We do not know whether or not, conversely, every outer mulitplier $a$ of $H^2_d$ has a (left or right) {\em outer} free lift.)

\begin{proof}[Proof of Theorem~\ref{thm:commutative-main}]
  Given $h\in H^2_d$ there exists a free lift to a function
  $H\in\mathfrak F^2_d$ with $\|H\| =\|h\|$. Applying Theorem~\ref{thm:free-main}, there is a canonical pair $(A, B)$ such that $H(X)=B(X)A(X)^{-1}$ for all $X$
  in the free unit ball. Passing to commuting arguments we get
  a column contraction $a(z)=A(z), b(z)=B(z)$ so that $h(z) =
  \frac{b(z)}{a(z)}$, and $a$ is outer by Lemma~\ref{lem:free-outer-implies-outer}. Finally, since $A^{-1}$ belongs to $\mathfrak F^2_d$ and the symmetrization map is contractive, we conclude that $1/a\in H^2_d$ and $\|1/a\|^2\leq \|A^{-1}\|^2 = 1+\|H\|^2 =1+\|h\|^2$.
\end{proof}

It is possible to extend Theorem~\ref{thm:free-main} and its commutative counterpart Theorem~\ref{thm:commutative-main} to sequences $\{H_i\}$ or $\{h_i\}$ of functions, representing them as Smirnov functions with a common denominator (compare \cite[Theorem 1.1]{AHMR-2017-2}). We have:
\begin{prop}
  If $\{H_i\}$ is a sequence in $\mathfrak F^2_d$ with $\sum_{i=1}^\infty \|H_i\|^2<\infty$, then there exists a free function $A(X)$ and a sequence $B_i(X)$ of free functions such that:
  \begin{itemize}
  \item[i)] $A$ is left free outer,
\item[ii)] for each $i$, $H_i(X)=B_i(X)A(X)^{-1}$,   
\item[iii)] the column 
    \begin{equation}
      \begin{pmatrix} A \\ B_1 \\ B_2 \\ \vdots \end{pmatrix}
    \end{equation}
is a left isometric multiplier, and
  
\item[iv)]  $A(X)^{-1} \in\mathfrak F^2_d$ and $\|A^{-1}\|_{\mathfrak F^2_d}^2 =1+\sum_{i=1}^\infty \|H_i\|_{{\mathfrak F}^2_d}^2.$
  \end{itemize}
\end{prop}
\begin{proof}
  The proof is identical to that given in the single function case; one need only replace the densely defined left multiplier $F\to  HF$ by the densely defined left column multiplier
  \begin{equation}
    F\to \begin{pmatrix} H_1F \\ H_2F \\ \vdots \end{pmatrix}. 
  \end{equation}
\end{proof}
As before, the following corollary is immediate, by applying the above proposition to a norm-preserving free lift $\{H_i\}$ of the sequence $\{h_1\}$:
\begin{cor}\label{cor:sequences-h2d}
  If $\{h_i\}$ is a sequence in $H^2_d$ with $\sum_{i=1}^\infty \|h_i\|^2<\infty$, then there exist multipliers $a, b_1, b_2, \dots \in Mult(H^2_d)$ such that:
  \begin{itemize}
  \item[i)] $a$ is outer, 
\item[ii)] For each $i$, $h_i=\frac{b_i}{a}$.   
\item[iii)] the column 
    \begin{equation}
      \begin{pmatrix} a \\ b_1 \\ b_2 \\ \vdots \end{pmatrix}
    \end{equation}
is a contractive multiplier, and
  \item[iv)] $1/a\in H^2_d$ and $\|1/a\|_{H^2_d}^2 \leq 1+\sum_{i=1}^\infty \|h_i\|_{H^2_d}^2$. 
  \end{itemize}
\end{cor}

Using the fact that $H^2_d$ is the ``universal model'' for CNP spaces, we obtain:
\begin{cor}\label{cor:sequences-cnp}
  The statement of Corollary~\ref{cor:sequences-h2d} holds with $\mathcal H(k)$ in place of $H^2_d$, for any normalized CNP kernel $k$. 
\end{cor}
\begin{proof}
 Given the sequence $(h_i)\subset \mathcal H_E$, (recall Equation (\ref{HE}) for the definition of $\mathcal{H} _E$), we extend isometrically to $H^2_d$, invoke Corollary~\ref{cor:sequences-h2d}, and restrict the multipliers $a, b_i$ we obtain back to $E$.  In light of the foregoing discussion, the only thing that remains to be checked is that if $a$ is outer in $Mult(H^2_d)$, then its restriction to $E$ is outer in $Mult(\mathcal H_E)$. We prove the contrapositive. Let $\varphi=a|_E$.  If $M_\varphi$ is not outer for $Mult(\mathcal H_E)$, that is, if $\mathrm{Ran} (M_\varphi)$ is not dense in $\mathcal H_E$, then $M_\varphi^*$ is not injective, i.e. there exists a nonzero $f\in\mathcal H_E$ with $M_\varphi^* f=0$. By the definition of $\mathcal H_E$, it is invariant for adjoints of multipliers $M_b^*$ on $H^2_d$ (indeed, $\mathcal H_E$ is a span of kernel functions in $H^2_d$, which are eigenfunctions for all the $M_b^*$). Therefore if $a$ is any extension of $\varphi$ to $Mult(H^2_d)$, it follows that also $M_a^*f=0$ in $H^2_d$, so $M_a^*$ is not injective and $a$ is not outer in $Mult(H^2_d)$. 
\end{proof}
Perhaps unfortunately, in general starting with $h\in H^2_d$,
the pair $(a,b)$ obtained from Theorem~\ref{thm:commutative-main} will depend on the choice
of free lift $H$, as the following examples show. 

\subsection{Examples.} Consider the function $h(z_1, z_2)= z_1+z_1z_2$
in $H_2^2$. We will compute the representations $h=b/a$ provided by
Theorem~\ref{thm:commutative-main}, arising from three different choices of free lift
\begin{equation}
  X_1+X_2X_1, \quad X_1+X_1X_2, \quad X_1+\frac12 (X_1X_2+X_2X_1).
\end{equation}
(The third is the norm-preserving lift.) In each case we can employ the following general strategy. When $H$ is
in fact a bounded multiplier itself, if $A,B$ is an inner-outer pair
representing $H$ then from the inner property and the fact that $B=HA$
we obtain
\begin{equation}
 I =  M_A^*M_A +M_B^*M_B  =M_A^*M_A+M_{HA}^*M_{HA} = M_A^*M_A +M_A^*M_H^*M_HM_A
\end{equation}
Thus, conversely, if we can find a bounded, invertible multiplier
$C$ with the property that 
\begin{equation}
  \label{eq:2}
  M_C^*M_C = I+M_H^*M_H,
\end{equation}
then putting $A=C^{-1}$ and $B=HA$ gives an inner-outer pair
representing $H$, which is of course unique (up to a unimodular
constant). In fact, (still assuming $H$ is a bounded multiplier), such
a $C$ will always exist, by a result of Popescu \cite[Corollary 1.4]{popescu-2006-entropy}. In these
examples we will be able to compute it explicitly. 

For $H(X) = X_1+X_2X_1$ we can first observe that left multiplication
by the words $1$ and $21$ respectively are isometries with orthogonal
ranges, so $M_H$ is itelf twice an isometry, so $I+M_H^*M_H =3I$. Thus
$A=3^{-1/2}$ is constant, so we get
\begin{equation}
  a(z) = 3^{-1/2},\quad b(z) = 3^{-1/2} (z_1+z_1z_2).   
\end{equation}
For $H(X)=X_1+X_1X_2 =X_1(I+X_2)$ we have
\begin{equation}
  I+M_H^*M_H =   I +M_{I+X_2}^*M_{X_1}^*M_{X_1} M_{1+X_2} =
  3+M_{X_2}^* +M_{X_2}.
\end{equation}
By the Fejer-Riesz factorization theorem, there exist constants $c_0,
c_1$ such that for all $|z|=1$
\begin{equation}
  3+z+\overline{z} = |c_0+c_1z|^2:=|c(z)|^2.
\end{equation}
(In fact one may verify that $c_0=\frac12 (\sqrt{5}+1)$ and
$c_1=\frac12 (\sqrt{5}-1)$ work.) So by the reasoning above we obtain
\begin{equation}
  a(z) = (c_0 +c_1z_2)^{-1}, \quad b(z) = z_1 (1+z_2)(c_0+c_1z_2)^{-1}.  
\end{equation}

By similar arguments we obtain for the lift $H(X)=X_1+\frac12
(X_1X_2+X_2X_1)$
\begin{equation}
  a(z) =  (d_0+d_1z_2)^{-1}, \quad b(z) =   z_1 (1+z_2)(d_0+d_1z_2)^{-1}. 
\end{equation}
with $d_0=\frac{1}{2\sqrt{2}} (\sqrt{7}+\sqrt{3})$, $d_1= \frac{1}{2\sqrt{2}} (\sqrt{7}-\sqrt{3})$.

\bibliographystyle{plain} 
\bibliography{WP} 

\end{document}